\documentclass[reqno]{amsart}

\usepackage{amsmath,amssymb,amsthm}

\usepackage{tikz}

\usepackage{caption}
\usepackage{graphicx}
\usepackage{graphics}

\newtheorem*{thm}{Theorem}
\newtheorem*{lemma}{Lemma}

\begin{document}

\title[]{On a problem involving unit fractions}

\author[]{Stefan Steinerberger}
\address{Department of Mathematics, University of Washington, Seattle, WA 98195, USA}
\email{steinerb@uw.edu}

\begin{abstract} Erd\H{o}s and Graham proposed to determine the number of subsets $S \subseteq \left\{1,2,\dots,n\right\}$ with $\sum_{s \in S} 1/s = 1$ and asked, among other things, whether that number could be as large as $2^{n - o(n)}$. We show that the number of subsets $S \subseteq \left\{1,2,\dots,n\right\}$  with $\sum_{s \in S} 1/s \leq 1$ is smaller than $2^{0.93n}$. 
\end{abstract}

\thanks{The author was partially supported by the NSF (DMS-212322).}

\maketitle

\section{The Result}
 A question of Erd\H{o}s and Graham \cite{erd} is as follows: how many subsets of the first $n$ integers $S \subseteq \left\{1,2,3,\dots, n\right\}$ describe unit fractions that sum to one, meaning subsets for which $$ \sum_{s \in S} \frac{1}{s} = 1~?$$
They ask whether the correct number might grow like $2^{c n}$ for some $0<c<1$ or whether it might be even as large as $2^{n - o(n)}$. The problem is also listed as Problem $\# 297$ in the list of Erd\H{o}s problems curated by Bloom \cite{bloom}. 
\begin{thm} We have, for $n$ sufficiently large,
$$ \# \left\{ S \subseteq \left\{1,2,\dots, n\right\} : \sum_{s \in S} \frac{1}{s} \leq 1 \right\} \leq 2^{0.93n}.$$
\end{thm}
This set of subsets trivially contains all subsets of $\left\{n/2, n/2+1, \dots, n\right\}$ and thus the size is at least $2^{n/2}$. The question about the size of the smaller set of subsets for which $ \sum_{s \in S} \frac{1}{s} = 1$ remains of interest, presumably that number is much smaller.\\

\textbf{Note added.} After posting the paper on arXiv, I was informed by Zachary Chase of a mathoverflow post\footnote{https://mathoverflow.net/questions/281124/sets-of-unit-fractions-with-sum-leq-1} where the very same problem was raised in 2017 by user `Mikhail Tikhomirov' who was interested in it for completely different reasons and did no connect it to the Erd\H{o}s problem. User `Lucia', also in 2017, provided a solution that is similar to mine in spirit (a Bernstein-type trick turning sums into products) and even ends up with a slightly better constant. Additionally, in the meantime, two independent teams established the sharp constant.  I will keep this preprint on the arXiv for archival reasons but it is not submitted to any journal.

\section{Proof of the Theorem}
\subsection{Rephrasing the problem.}
Instead of counting the number of subsets, we introduce indicator variables $\delta_i \in \left\{0, 1\right\}$ and are interested in whether
$$ \sum_{i=1}^{n} \frac{\delta_i}{i} \leq 1.$$
The question is now for how many of the $2^n$ choices of $(\delta_1, \dots, \delta_n) \in \left\{0,1\right\}^n$ this inequality is satisfied. We will rephrase the question in a more symmetric way.
Writing 
$ \delta_i =(1 + \varepsilon_i)/2,$
where $\varepsilon_i \in \left\{-1,1\right\}$, we get
$$ \sum_{i=1}^{n} \frac{1 + \varepsilon_i}{2i}  =  \frac{1}{2} \sum_{i=1}^{n} \frac{1}{i}  + \frac{1}{2} \sum_{i=1}^{n} \frac{\varepsilon_i}{i} = \frac{H_n}{2} + \frac{1}{2} \sum_{i=1}^{n} \frac{\varepsilon_i}{i},$$
with $H_n = 1 + 1/2 + \dots + 1/n$ being the $n-$th harmonic number. The question can thus be equivalently written as follows: for how many $\varepsilon \in \left\{-1,1\right\}^n$ do we have
$$  \sum_{i=1}^{n} \frac{\varepsilon_i}{i} \leq 2 -H_n \qquad \mbox{or, by symmetry,} \qquad  \sum_{i=1}^{n} \frac{\varepsilon_i}{i} \geq H_n-2 ~?$$

\subsection{Hoeffding.} We can interpret this new sum as a random walk with decreasing step size. The usual out-of-the-box deviation estimates appear to be not quite delicate enough: for example, we have
$$\mathbb{E} \left(\sum_{i=1}^{n} \frac{\varepsilon_i}{i} \right) = 0 \qquad \mbox{and} \qquad \mathbb{V}  \left(\sum_{i=1}^{n} \frac{\varepsilon_i}{i} \right) = \sum_{i=1}^{n} \frac{1}{i^2} \leq \frac{\pi^2}{6}.$$
Comparing with a Gaussian with the same parameters would, in the best case, only give something along the lines of
$ e^{-c (\log{n} )^2}$. We will instead go through the main idea behind the proof of Hoeffding's inequality and then estimate things in a manner more adapted to the problem at hand. For each $x>0$ and $t > 0$
\begin{align*}
\mathbb{P} \left(  \sum_{i=1}^{n} \frac{\varepsilon_i}{i}  \geq t\right) =\mathbb{P} \left( x  \sum_{i=1}^{n} \frac{\varepsilon_i}{i}  \geq x t\right)  =  \mathbb{P} \left( \exp\left(  - x t + x\sum_{i=1}^{n} \frac{\varepsilon_i}{i}\right) \geq 1\right).
\end{align*}
For any nonnegative random variable $X$ one has
$ \mathbb{P}(X \geq 1) \leq \mathbb{E}X$
and thus
$$ \mathbb{P} \left(  \sum_{i=1}^{n} \frac{\varepsilon_i}{i}  \geq t\right) \leq \mathbb{E}  \exp\left(  - x t + x\sum_{i=1}^{n} \frac{\varepsilon_i}{i}\right) = e^{-xt} \cdot \mathbb{E}  \exp\left( \sum_{i=1}^{n} \frac{\varepsilon_i x}{i}\right) $$
The exponential function sends sum to products, the $\varepsilon_i$ are independent random variables and $\mathbb{E}(XY) = (\mathbb{E}X)(\mathbb{E}Y)$ whenever $X$ and $Y$ are independent. Thus
\begin{align*}
 \mathbb{P} \left(  \sum_{i=1}^{n} \frac{\varepsilon_i}{i}  \geq t\right) \leq e^{- x t} \cdot  \prod_{i=1}^{n} \left( \mathbb{E} e^{x \varepsilon_i/i} \right) = e^{- x t} \cdot \prod_{i=1}^{n} \left( \frac{e^{x/i} + e^{-x/i}}{2} \right) 
\end{align*}

\subsection{Product term.} The next step is an estimate for the product term.
\begin{lemma}
For $x>0$ and all $2 \leq m \leq n$, we have
$$\prod_{i=1}^{n} \left( \frac{e^{x/i} + e^{-x/i}}{2} \right) \leq  \left(\frac{1+ e^{-2x/m}}{2}\right)^m \exp \left( x \cdot H_m + \frac{x^2}{2m} \right).$$
\end{lemma}
\begin{proof}
We split the product into small and large values of $i$
\begin{align*}
\prod_{i=1}^{n} \left( \frac{e^{x/i} + e^{-x/i}}{2} \right) = \prod_{1\leq i \leq m}^{} \left( \frac{e^{x/i} + e^{-x/i}}{2} \right) \cdot \prod_{m < i \leq n}^{} \left( \frac{e^{x/i} + e^{-x/i}}{2} \right).
\end{align*}
For small values of $i$, we argue
\begin{align*}
 \prod_{1 \leq i \leq m}^{} \left( \frac{e^{x/i} + e^{-x/i}}{2} \right) &= \frac{1}{2^m}  \prod_{1 \leq i \leq m}^{}  e^{x/i}  \left( 1+ e^{-2x/i} \right) \leq \frac{1}{2^m}  \prod_{1 \leq i \leq m}^{}  e^{x/i}  \left( 1+ e^{-2x/m} \right) \\
 &= \left(\frac{1+ e^{-2x/m}}{2}\right)^m\prod_{1 \leq i \leq m}^{}  e^{x/i} = \left(\frac{1+ e^{-2x/m}}{2}\right)^m e^{x \cdot H_m}.
\end{align*}
For large $i$, we use $(e^x + e^{-x})/2 \leq e^{x^2/2}$ to bound
\begin{align*}
\prod_{m < i \leq n}^{} \left( \frac{e^{x/i} + e^{-x/i}}{2} \right) &\leq \prod_{m < i \leq n}^{} \exp\left(  \frac{x^2}{2i^2}\right)= \exp\left(  \frac{x^2}{2} \sum_{m < i \leq n} \frac{1}{i^2} \right). 
\end{align*}
The result then follows from
$$ \sum_{m < i \leq n} \frac{1}{i^2} \leq \sum_{ i = m+1}^{\infty} \frac{1}{i^2} \leq \int_{m}^{\infty} \frac{dx}{x^2} = \frac{1}{m}.$$
\end{proof}

\subsection{Conclusion}  
We have, for all $t, x > 0$ and all integers $2 \leq m \leq n$
 \begin{align*}
  \mathbb{P} \left(  \sum_{i=1}^{n} \frac{\varepsilon_i}{i}  \geq t\right) \leq e^{-x t} \left(\frac{1+ e^{-2x/m}}{2}\right)^m \exp \left( x \cdot H_m + \frac{x^2}{2m} \right).
  \end{align*}
The relevant value is $t = H_n - 2$. If $m \leq n/8$, then
$$ H_n - H_m \geq H_n - H_{n/8} \geq \int_{n/8}^{n} \frac{dx}{x}  +\mathcal{O}\left(\frac1n\right)= \log(8) +\mathcal{O}\left(\frac1n\right) \sim 2.07 > 2 $$
and thus $(H_n - H_m -2)m > 0$. We set $x = (H_n - H_m - 2)m$. Some computation shows that then
\begin{align*}
  \mathbb{P} \left(  \sum_{i=1}^{n} \frac{\varepsilon_i}{i}  \geq H_n -2\right) &\leq \exp\left( - \frac{m}{2} (H_n - H_m - 2)^2 \right)  \cdot  \left(\frac{1+ e^{- 2(H_n - H_m -2)}}{2}\right)^m.
\end{align*}
With $ H_n = \log{n} + \gamma + \mathcal{O}(n^{-1})$
and the ansatz $m=  cn$ for some $0 < c < 1/8$, 
\begin{align*}
   \mathbb{P} \left(  \sum_{i=1}^{n} \frac{\varepsilon_i}{i}  \geq H_n -2\right) &\leq \exp\left[ - \frac{cn}{2} \left( \log\left(\frac{1}{c}\right) - 2 +\mathcal{O}\left(\frac1n\right)\right)^2 \right] \\
   &\cdot \exp\left[c n \log\left( \frac{1+ e^{-2 (\log\left( \frac{1}{c} \right) -2+\mathcal{O}\left(\frac1n\right))}}{2}  \right) \right].
  \end{align*}
  
The function
$$ f(c) =- \frac{c}{2} \left( \log\left(\frac{1}{c}\right) - 2\right)^2 + c \log\left( \frac{1+ e^{-2 (\log\left( \frac{1}{c} \right) -2)}}{2}  \right)$$
satisfies $f(0.0384235) \leq -0.0541$. Since $e^{-0.054} \leq 2^{-0.07}$, the result follows.\\

\end{document}